\documentclass{amsproc}

\usepackage{}

\newtheorem{theorem}{Theorem}[section]
\newtheorem{lemma}[theorem]{Lemma}

\theoremstyle{definition}
\newtheorem{definition}[theorem]{Definition}

\newtheorem{problem}{Problem}

\theoremstyle{remark}

\numberwithin{equation}{section}

\begin{document}

\title[The Radon transform on $SO(3)$] {The Radon transform on $SO(3)$: motivations, generalizations, discretization}\footnote{Accepted and will appear as a book chapter in a volume of "Contemporary Mathematics"}
\dedicatory{Dedicated to S. Helgason on his 85-th Birthday}


\author{Swanhild Bernstein}
\address{TU Bergakademie Freiberg, Institute of Applied Analysis, Germany}
\curraddr{}
\email{swanhild.bernstein@math.tu-freiberg.de}
\thanks{}

\author{Isaac Z. Pesenson }
\address{ Department of Mathematics, Temple University,
 Philadelphia,
PA 19122, USA}
\curraddr{}
\email{pesenson@temple.edu}
\thanks{The author was supported in
part by the National Geospatial-Intelligence Agency University
Research Initiative (NURI), grant HM1582-08-1-0019.}


\date{}

\begin{abstract}
In this paper we consider a version of the Radon transform $\mathcal{R}$ on the group of rotations  $SO(3)$ and closely related crystallographic $X$-ray transform $P$  on $SO(3)$.  We  compare the  Radon transform $\mathcal{R}$ on $SO(3)$ and the totally geodesic $1$-dimensional Radon transform on $S^{3}$.  An exact reconstruction formula for bandlimited function $f$ on $SO(3)$ is introduced, which uses only a finite number of samples of the Radon transform $\mathcal{R} f$.
\end{abstract}

\maketitle

\section{Introduction}

In this paper we consider a version of the Radon transform $\mathcal{R}$ on the group of rotations  $SO(3)$ and closely related crystallographic $X$-ray transform $P$  on $SO(3)$\footnote{In \cite {P} the same transform $\mathcal{R}$ was termed  as the Funk transform.}   We show that both of these transforms 
 naturally appear in texture analysis, i.e. the analysis of preferred crystallographic orientation.  Although we discuss only applications to texture analysis both transforms have other applications as well.

 The structure of the paper is as follows. 
 In section 2 we start with motivations and applications. In section 3  we  collect some basic facts about Fourier analysis on compact Lie groups. In section 4 we introduce and analyze an analog of $\mathcal{R}$  for general compact Lie groups.  In the case of the group $SO(n+1)$  we compute image $\mathcal{R} (\mathcal{W})$ where $\mathcal{W}$ is the span of Wigner polynomials in $SO(n+1)$.              In section 5 we give a detailed analysis of the  Radon transform $\mathcal{R}$  on $SO(3)$. In section 6 we describe relations between $S^{3},\>SO(3)$ and $S^{2}\times S^{2}$ and we compare the  Radon transform $\mathcal{R}$ on $SO(3)$ and the totally geodesic $1$-dimensional  Radon transform on $S^{3}$. In section 7 we show non-invertibility of the crystallographic $X$-ray transform $P$. In section  8  we describe  an \textit{exact} reconstruction formula for bandlimited function $f$ on $SO(3)$, which uses only a \textit{finite} number of samples of the Radon transform $\mathcal{R}f$.  Some auxiliary results for this section are collected in Appendix.

The Radon transform on $SO(3)$ has recently attracted attention of many mathematicians. In addition to articles, which will be mentioned in our paper later we also refer to \cite{BPS}, \cite{hielscher}, \cite{HPPSS}, \cite{K}, \cite{Meister}, \cite{P}.

\section{Texture goniometry}
A first mathematical description of the inversion problem in texture analysis was given in \cite{HJB1} and \cite{HJB2}.
Let us  recall the basics of texture analysis and texture goniometry 
(see \cite{bernstein/schaeben}
and \cite{BHS}).
Texture analysis is the analysis of the statistical distribution of orientations
of crystals within a specimen of a polycrystalline material, which could be
metals or rocks. A crystallographic orientation is a set of crystal
symmetrically equivalent rotations between an individual crystal and the
specimen.

The main objective is to determine \textit{orientation probability density function $f$ (ODF)} representing the probability 
law of random orientations of crystal grains by volume.

In X-ray diffraction experiments, the \textit{orientation density function} $f$ (ODF) that represents the probability law of random orientations of crystal grains cannot be measured directly. Instead, by using a texture goniometer the \textit{pole density function (PDF) }
$Pf(x,y)$ can be sampled. $Pf(x,y)$ represents probability that a fixed crystal direction $x\in S^{2}$
or its antipodal $-x$  statistically coincides with the specimen direction $y\in S^{2}$ due to 
Friedel's law in 
crystallography \cite{F}. 

To define the pole density function 
$Pf(x,y)$ some preliminaries are necessary. The group  of rotations $SO(3)$ of $\mathbb{R}^{3}$ consists of $3\times 3$ real matrices $U$ such that  $U^TU = I,\>\>\ {\rm det\,}U = 1$.
It is known that any $g\in SO(3)$ has a unique representation of the form
$$ 
g = Z(\gamma)X(\beta)Z(\alpha),\ 0\leq \beta \leq \pi,\ 0\leq \alpha,\,\gamma < 2\pi,
$$
 where 
$$ Z(\theta) = \left(\begin{array}{ccc} \cos \theta & -\sin \theta & 0 \\ \sin \theta & \cos \theta & 0 \\ 0 & 0 & 1 \end{array}\right), \quad \mbox X(\theta) = \left(\begin{array}{ccc} 1 & 0 & 0 \\ 0 & \cos \theta & -\sin \theta  \\ 0 & \sin \theta & \cos \theta \end{array}\right) 
$$
are  rotations about the $Z$- and $X$-axes, respectively. In the coordinates $\alpha, \beta, \gamma$, which are known as Euler angles, the Haar measure  of the group $SO(3)$ is given as (see \cite{Naimark})
$$
dg= \frac{1}{8\pi^2} \sin\beta d\alpha\,d\beta\,d\gamma. 
$$ 
In other words the following formula holds:
$$ \int_{SO(3)} f(g)\,dg = \int_0^{2\pi}\int_0^{\pi}\int_0^{2\pi} f(g(\alpha,\,\beta,\,\gamma))\frac{1}{8\pi^2} \sin\beta d\alpha\,d\beta\,d\gamma.
$$ 
First, we introduce Radon transform $\mathcal{R} f$ of a smooth function $f$ defined on $SO(3)$. If $S^{2}$ is the standard unit sphere in $\mathbb{R}^{3}$ , then for a pair $(x,y)\in S^{2}\times S^{2}$ the value of the  Radon transform $\mathcal{R} f$ at $(x,y)$ is defined by the formula
\begin{equation}\label{RRR}
 (\mathcal{R} f)(x,y)   = \frac{1}{2\pi} \int_{\{g\in SO(3): x=gy\}} f(g) d\nu_g = 
 $$
 $$
 4\pi \int_{SO(3)} f(g)\delta_y(g^{-1}x) dg = (f*\delta_y)(x),\>\>\> (x,y)\in S^{2}\times S^{2},
\end{equation}
where $d\nu_g = 8\pi ^2 dg, $  and  $\delta_{y}$ is the measure concentrated on the set of all $g\in SO(3)$ such that $x=gy$.

The pole density function  $Pf$ or \textit{crystallographic X--ray transform} of an orientation density function $f$  is an even function on $S^{2}\times S^{2}$, which is defined by the following  formula
\begin{equation}\label{PDF}
 Pf(x,y) = \frac{1}{2}(\mathcal{R} f(x,y)+\mathcal{R} f(-x,y)),\>\>\>(x,y)\in S^{2}\times S^{2}.
 \end{equation}
Note, that  since  ODF $f$ is a probability  density it has to have the following properties: 

(a) $f(g)\geq 0,$ 

(b) $\int_{SO(3)} f(g) dg = 1.$

In what follows we will discuss inversion of the crystallographic $X$-ray  transform $Pf$ and the Radon transform $\mathcal{R} f$.

First we formulate what can be called analytic reconstruction problem.
\begin{problem}
Reconstruct the ODF $f(g),\, g\in SO(3),$ from PDF $Pf(x,y),\,
x,y\in S^2.$ 
\end{problem}
It will be shown in section 7 that this problem is unsolvable in general since the mapping
$
f\rightarrow Pf
$
has a non-trivial kernel.

\begin{problem}
Reconstruct  $f(g),\, g\in SO(3),$  from all  $\mathcal{R}f(x,y),\, x,y\in S^2.$ 
\end{problem}

An explicit solution to this problem will be given in section 5.

In practice only a finite number of pole figures  $P(x,y),\,
x,y\in S^2,$ can be measured.  Therefore the real life  reconstruction problem is the following.

\begin{problem}

Using a finite number of pole figures $P(x_i,y_j),\,
x_{i},y_{j}\in S^2,\ i=1,\,\ldots,\,n,\,j=1,\,\ldots ,\, m,  $ find a function $f$ on $SO(3)$, which would satisfy (in some sense) equations (\ref{PDF}) and conditions (a) and (b).
\end{problem}
An approximate solution to this problem  in terms of  Gabor frames  was found in \cite{CKT}.

The corresponding discrete problem for $\mathcal{R} f$ can be formulated as follows.

\begin{problem}
Reconstruct  $f(g),\, g\in SO(3),$  from a finite number of samples  $\mathcal{R} f(x_{j},y_{j}),\, x_{j},y_{j}\in S^2,\>j=1,...,m$. 
\end{problem}

This problem will be solved in section 8 for bandlimited functions on $SO(3)$.  We were able to obtain an \textit{exact} reconstruction formula for bandlimited functions, which uses only a \textit{finite} number of samples of their Radon transform.  Another approach to this problem which uses the so-called generalized splines on $SO(3)$ and $S^{2}\times S^{2}$ was developed in our paper \cite{BEP}.

In section 4  we suggest a new type of  Radon transform associated with a pair $(\mathcal{G}, \>\mathcal{H})$ where $\mathcal{G}$ is a  compact Lie group and $\mathcal{H}$ its closed subgroup. This definition appeared for the first time  in our paper  \cite{BEP}.
Namely,  for every continuous function $f$ on $\mathcal{G}$ the corresponding Radon transform is defined by the formula 
 \begin{equation}\label{RD-100}
 \mathcal{R} f(x,y) = \int_{\mathcal{H}} f(xhy^{-1}) \,dh,\quad x,\,y \in \mathcal{G}. 
  \end{equation}
 \begin{problem}
 Determine domain and
 range for the Radon transform $\mathcal{R}$
 \end{problem}
Some partial solutions to this problem are given in section 4.
In section 3 we recall basic facts about Fourier analysis on compact Lie groups.  In section 6 we compare  crystallographic $X$-ray transform  on $SO(3)$ and Funk  transform on $S^{3}$.  
In Appendix 9 we briefly explain the  major ingredients of the proof of our Discrete Inversion Formula which is obtained in section 8.

\section{Fourier Analysis on compact groups}\label{F}

Let $\mathcal{G}$ be a compact Lie group. A unitary representation of $\mathcal{G}$ is a continuous group 
homomorphism $\pi$: $\mathcal{G}\to U(d_{\pi})$ of $\mathcal{G}$ into the group of unitary matrices of a 
certain dimension $d_{\pi}$. 
Such  representation is irreducible if $\pi(g)M=M\pi(g)$ for all $g\in \mathcal{G}$ and some 
$M\in \mathbb{C}^{d_{\pi}\times d_{\pi}}$ implies $M=cI$, where $I$  is  the identity matrix. 
Equivalently, $\mathbb{C}^{d_{\pi}}$ does not have non-trivial $\pi$-invariant subspaces 
$V\subset \mathbb{C}^{d_{\pi}}$ with $\pi(g)V \subset V$ for all $g\in \mathcal{G}.$ 
Two representations $\pi_1$ and $\pi_2$ are equivalent, if there exists an invertible matrix 
$M$ such that $\pi_1(g)M=M\pi_2(g)$ for all $g\in \mathcal{G}$.

Let $\hat{\mathcal{G}}$ denote the set of all equivalence classes of irreducible representations. This set parameterizes an orthogonal decomposition of the Hilbert space $L^2(\mathcal{G})$ constructed with respect to the normalized Haar measure. Let $\{e_j\}$ be an orthonormal basis for the  unitary matrices $U(d_{\pi})$ of dimension $d_{\pi}.$ Then for any unitary representation of $\mathcal{G}$ the $\pi_{ij}(g)=\langle \pi(g)e_j,\,e_i\rangle$ are called matrix elements of $\pi.$ We denote the linear span of the matrix elements of $\pi$ by $H_{\pi}.$
\begin{theorem}[Peter-Weyl, \cite{vilenkin1}] Let $\mathcal{G}$ be a compact Lie group. 
  Then the following statements are true.
\begin{description}
	\item[a] The Hilbert space $L^2(\mathcal{G})$ decomposes into 
	the orthogonal direct sum
	\begin{eqnarray}
		L^2(\mathcal{G}) = \bigoplus_{\pi\in \hat{\mathcal{G}}} H_{\pi}
	\end{eqnarray}
	\item[b] For each irreducible representation $\pi\in \hat{\mathcal{G}}$ the orthogonal projection  \\
	 $L^2(\mathcal{G})\to H_{\pi}$ is given by
	\begin{eqnarray}
		f \mapsto d_{\pi} \int_{\mathcal{G}} f(h)\chi_{\pi}(h^{-1}g)\,dh = d_{\pi}\,f*\chi_{\pi},
	\end{eqnarray}
	in terms of the character $\chi_{\pi}(g)={\rm trace}(\pi(g))$ of the representation and $dh$
	 is the normalized Haar measure.
\end{description}
\end{theorem}
We will denote the matrix $M$ in the equation $f*\chi_{\pi}={\rm trace}(\pi(g)M)$ as the Fourier coefficient $\hat{f}(\pi)$ of $f$ at the irreducible representation $\pi$. The Fourier coefficient can be calculated as
$$ \hat{f}(\pi) = \int_{\mathcal{G}} f(g)\pi^*(g)\,dg,\>\> \pi \in \hat{\mathcal{G}} .$$
The inversion formula (the Fourier expansion) is then given by
$$ f(g) = \sum_{\pi\in\hat{\mathcal{G}}} d_{\pi}\,{\rm trace}(\pi(g)\hat{f}(\pi)). $$
If we denote by $||M||^2_{HS}={\rm trace}(M^*M)$ the Frobenius or Hilbert-Schmidt norm of a matrix $M,$ then the following Parseval identity is true.
\begin{theorem}[Parseval identity] Let $f\in L^2(\mathcal{G}).$ Then the matrix-valued Fourier coefficients $\hat{f}\in \mathbb{C}^{d_{\pi}\times d_{\pi}}$ satisfy
\begin{eqnarray}
||f||^2 = \sum_{\pi\in\hat{\mathcal{G}}} d_{\pi}\,||\hat{f}(\pi)||^2_{HS} . \label{parseval_id}
\end{eqnarray}
\end{theorem}
On the group $\mathcal{G}$ one defines the convolution of two integrable functions $f,\,r\in L^1(\mathcal{G})$ as
$$ f*r(g) = \int_{\mathcal{G}} f(h)r(h^{-1}g)\,dh . $$
Since $f*r\in L^1(\mathcal{G}),$ the Fourier coefficients are well-defined and they satisfy
\begin{theorem}[Convolution theorem on $\mathcal{G}$] Let $f,\,r\in L^1(\mathcal{G})$ then $f*r\in L^1(\mathcal{G})$ and
	$$ \widehat{f*r}(\pi) = \hat{f}(\pi)\hat{r}(\pi). $$
\end{theorem}
The group structure gives rise to the left and right translations $T_gf\mapsto f(g^{-1}\cdot)$ and $T^gf\mapsto f(\cdot g)$ of functions on the group. A simple computation shows
$$ \widehat{T_gf}(\pi) = \hat{f}(\pi)\pi^*(g) \quad \mbox{and}\quad \widehat{T^gf}(\pi) = \pi(g)\hat{f}(\pi). 
$$
These formulas are direct consequences of the definition of the Fourier transform. 

The Laplace-Beltrami operator $\Delta_{\mathcal{G}}$ of an invariant metric on the group $\mathcal{G}$ is bi-invariant, i.e. commutes 
with all $T_g$ and $T^g.$ Therefore, all its eigenspaces are bi-invariant subspaces of $L^2(\mathcal{G}).$ 
As $H_{\pi}$ are minimal bi-invariant subspaces, each of them has to be the eigenspace of $\Delta_{\mathcal{G}}$ with the corresponding eigenvalue $-\lambda_{\pi}^2.$ Hence, we obtain
$$ \Delta_{\mathcal{G}}f = -\sum_{\pi\in\hat{\mathcal{G}}} d_{\pi}\,\lambda_{\pi}^2\,{\rm trace}(\pi(g)\hat{f}(\pi)). $$

\section{Problem 5: Radon transform on compact groups}\label{GRT}
\subsection{Radon transform}

In this section we discuss some basic properties on the Radon transform $\mathcal{R}f$  which was defined in (\ref{RD-100}).
\begin{theorem}[\cite{BEP}]
 The Radon transform (\ref{RD-100}) is invariant under right shifts of
 $x$ and $y,$ hence it maps functions on  $\mathcal{G}$ to functions on $\mathcal{G}/ \mathcal{H} \times \mathcal{G} / \mathcal{H} . $
\end{theorem}
\begin{proof} 
  First, we take the Fourier transform of $\mathcal{R} f$ with respect to the $x$ and let
  $y$  be fixed and regard $\mathcal{R}  f(x,\,y)$ as a function of $x\in \mathcal{G}$ only. Then
  $$ \widehat{\mathcal{R} f(\cdot ,\, y)}(\pi)= \pi_{\mathcal{H}}\pi^*(y)\hat{f}(\pi),\quad \pi\in
\hat{\mathcal{G}}. $$
It is easily seen that $\mathcal{R} f(x,\,y) $ is invariant under the projection $\mathbb{P}_{\mathcal{H}}$ 
and we obtain
$$ \mathcal{R} f(x\cdot h,\,y) = \mathcal{R}f(x,y)\quad \forall h\in \mathcal{H} .$$
 If we look at the Radon transform as a function in $y$ while the first argument
$x$ is fixed, we find
\begin{equation}
 \mathbb{P}_{\mathcal{H}} (\mathcal{R}f)(x,y)  = \int_{\mathcal{H}} \mathcal{R} f(x,\,yh)\,dh = 
 \int_{\mathcal{H}} \sum_{\pi\in \hat{\mathcal{G}}}
 d_{\pi} {\rm trace\,}(\hat{f}(\pi)\pi(x))\pi_{\mathcal{H}}\pi(h^{-1}y^{-1}) \,dh \\ 
 $$
 $$  
   = \sum_{\pi\in \hat{\mathcal{G}}} d_{\pi} {\rm trace\,}(\hat{f}(\pi)\pi(x))\pi_{\mathcal{H}}\pi ^*(y)
= (\mathcal{R}) f (x,\,y).    
 \end{equation} 
 Consequently, $\mathcal{R} f(x,\,y) $ is constant over fibers of the form $y\mathcal{H} $ and 
$$ \widehat{\mathcal{R} f(x ,\, \cdot )}(\pi)= \pi_{\mathcal{H}}\pi^*(x)\hat{f}(\pi),\quad \pi\in
\hat{\mathcal{G}}. $$
\end{proof}

The next Theorem is a refinement of the previous result.
\begin{theorem}[\cite{BEP}]
Let $\mathcal{H}$ be a subgroup of $\mathcal{G}$ which determines the Radon transform on $\mathcal{G}$
and let $\hat{\mathcal{G}_1} \subset \hat{\mathcal{G}}$ be the set of irreducible representations
with respect to $\mathcal{H}.$ Then for $f\in C^{\infty}(\mathcal{G})$ we have
$$ ||\mathcal{R} f||^2_{L^2(\mathcal{G} / \mathcal{H} \times \mathcal{G} / \mathcal{H})} = \sum_{\pi\in\hat{\mathcal{G}_1}} {\rm
rank\,}(\pi_{\mathcal{H}}) ||\hat{f} ||^2_{HS} . $$
\end{theorem}
\begin{proof}
  We expand $\mathcal{R} f(x,y)$ for fixed $y$ into a series with respect to $x$ and
  apply Parseval's theorem
 \begin{equation}
   ||\mathcal{R} f ||^2_{L^2(\mathcal{G} / \mathcal{H} \times \mathcal{G} / \mathcal{H})}  = \sum_{\pi\in\hat{\mathcal{G}}} d_{\pi}
\int_{\mathcal{G}} || \pi_{\mathcal{H}}\pi^*(y)\hat{f}(\pi)||^2_{HS} dy =$$
$$
 \sum_{\pi\in\hat{\mathcal{G}}} d_{\pi} \int_{\mathcal{G}}
{\rm trace\,}(\hat{f^*}(\pi)\pi(y)\pi_{\mathcal{H}}\pi^*(y)\hat{f}(\pi)) \,dy =$$
$$
  \sum_{\pi\in\hat{\mathcal{G}}} d_{\pi} 
{\rm trace\,}(\hat{f^*}(\pi)\left(\int_{\mathcal{G}}\pi(y)\pi_{\mathcal{H}}\pi^*(y)\,dy\right)\hat{f}(\pi))=
$$
$$ \sum_{\pi\in\hat{\mathcal{G}}} d_{\pi} 
{\rm trace\,}(\hat{f^*}(\pi)\left(\sum_{k=1}^{{\rm rank\,}\pi_{\mathcal{H}}}\int_{\mathcal{G}}\pi_{ik}(y)\overline{\pi_{kj}(y)}\,dy\right)_{i,j=1}^{d_{\pi}}\hat{f}(\pi))=
$$
$$ \sum_{\pi\in\hat{\mathcal{G}}} d_{\pi} 
{\rm trace\,}(\hat{f^*}(\pi)\frac{{\rm rank\,}\pi_{\mathcal{H}}}{d_{\pi}}Id\hat{f}(\pi))=
$$
$$
 \sum_{\pi\in\hat{\mathcal{G}}}   {\rm rank\,}\pi_{\mathcal{H}} {\rm trace\,}(\hat{f^*}(\pi)\hat{f}(\pi)) = \sum_{\pi\in\hat{\mathcal{G}_1}} {\rm
rank\,}(\pi_{\mathcal{H}}) ||\hat{f} ||^2_{HS} . 
 \end{equation}
\end{proof}

\subsection{The case $\mathcal{G} = SO(n+1),\>\>\> \mathcal{H} = SO(n)$}

 We start
with the orthonormal system of spherical harmonics $\mathcal{Y}_k^i \in C^{\infty}(S^n),\ k\in
\mathbb{N}_0,\ i=1,\ldots , d_k(n)$  normalized with respect to the
Lebesgue measure on $S^n.$ Obviously $\mathcal{H}_k = {\rm
span\,}\{\mathcal{Y}_k^i\}_{i=1}^{d_k(n)}.$ Then the Wigner polynomials on
$SO(n+1)$ $\mathcal{T}_k^{ij}(g),\ g\in SO(n+1)$ are given by
$$ \mathcal{T}_k^{ij}(g) =
\int_{S^n}\mathcal{Y}_k^i(g^{-1}x)\overline{\mathcal{Y}_k^j(x)}\, dx
 $$
and due to the orthogonality of the spherical harmonics
$$ \mathcal{Y}_k^i(g^{-1}x) = \sum_{j=1}^{d_k(n)}
\mathcal{T}_k^{ij}(g)\mathcal{Y}_k^j(x). $$
From these properties and the orthonormality of the spherical harmonics it easy
to see that the Wigner polynomials build an orthonormal system in
$L^2(SO(n+1)).$ Unfortunately, Wigner polynomials do not give all irreducible
unitary  representations of $SO(n+1)$ if $n>2.$ 

\begin{definition}
A unitary representation of a group $\mathcal{G}$ in a liner space $\mathcal{L}$ is said to be of class-1 relative subgroup $\mathcal{H}$ if $\mathcal{L}$ contains non-trivial vectors that are invariant with respect to $\mathcal{H}$.

\end{definition}

\begin{definition}
If in the space $\mathcal{L}$ of any representation of class-1 relative $\mathcal{H}$
there is only one normalized invariant vector, then $\mathcal{H}$ is called a massive
subgroup.
\end{definition}

\begin{lemma}[\cite{Vilenkin}, Chapter IX.2] $SO(n)$ is a massive subgroup of
 $SO(n+1).$ Furthermore, the family $\mathcal{T}_k,\,k\in \mathbb{N}_0,$ 
gives  all class-1 representations of $SO(n+1)$ with respect to $SO(n)$ up to equivalence.
\end{lemma}

For the following let $x_0$ be the base point of $SO(n+1)/SO(n)\sim S^n$ ($x_0$ is usually chosen to be the "north pole".) In this case the set of zonal spherical harmonics is one-dimensional and spanned by the Gegenbauer polynomials $\mathcal{C}_k^{(n-1)/2}(x_0^Tx).$ We recall some helpful and well known results.

\begin{lemma}[Addition theorem] For all $x,\,y\in S^n,\ k\in\mathbb{N}_0$ and $i=1,\ldots ,\, d_k(n)$ 
$$ \frac{\mathcal{C}_k^{(n-1)/2}(x^Ty)}{\mathcal{C}_k^{(n-1)/2}(1)} = \frac{|S^n|}{d_k(n)} \sum_{i=1}^{d_k(n)} \mathcal{Y}_k^i(x)\overline{\mathcal{Y}_k^j(y)}. $$
\end{lemma}

\begin{lemma}[Zonal averaging] 
$$ \int_{SO(n)} \mathcal{Y}_k^i(gx)\,dg = \frac{\mathcal{Y}_k^i(x_0)}{\mathcal{C}_k^{(n-1)/2}(1)} \mathcal{C}_k^{(n-1)/2}(x_0^Tx). $$
\end{lemma}

\begin{lemma}[Funk-Hecke formula] Let $f: [-1,\,1]\to \mathcal{C}$ be continuous. Then for all $i=1,\ldots ,d_k(n)$
$$ \int_{S^n} f(x^Ty)\mathcal{Y}_k^i(x)\,dx = \mathcal{Y}_k^i(y)\frac{|S^{n-1}|}{\mathcal{C}_k^{(n-1)/2}(1)} \int_{-1}^1f(t)\mathcal{C}_k^{(n-1)/2}(t)(1-t^2)^{n/2-1}\,dt. $$
\end{lemma}
Since we are interested in functions on $S^n$, which we obtain by the projection from $SO(n+1)$, we have to consider all irreducible representations of $SO(n+1)$ which do not have vanishing matrix coefficients under the projection $\mathbb{P}_{SO(n)}$. 
 These irreducible representations form the class-1 representations of $SO(n+1)$ with respect to $SO(n)$ and the projections are given by 
\begin{equation}
\mathbb{P}_{SO(n)}\mathcal{T}_k^{ij} =  \int_{SO(n)} \mathcal{T}^{ij}_k (g)\, dg  = \int_{S^n}\int_{SO(n)} \mathcal{Y}_k^i(g^{-1}x) \,dg\, \mathcal{Y}_k^j(x)\,dx =
$$
$$
\frac{\mathcal{Y}_k^i(x_0)}{\mathcal{C}_k^{(n-1)/2}}\int_{S^n}\mathcal{C}_k^{(n-1)/2}(x_0^Tx)\mathcal{Y}_k^i(x)\,dx =
$$
$$
 \frac{\mathcal{Y}_k^i(x_0)\mathcal{Y}_k^j(x_0)}{(\mathcal{C}_k^{(n-1)/2}(1))^2}|S^n|\int_{-1}^1 (\mathcal{C}_k^{(n-1)/2}(t))^2 (1-t^2)^{n/2-1}\,dt= 
 $$
 $$
 \frac{|S^n|}{d_k(n)} \mathcal{Y}_k^i(gx_0)\mathcal{Y}_k^j(x_0),
\end{equation}
due to the Funk-Hecke formula and the normalization of Gegenbauer polynomials.
We assume  that the basis of spherical harmonics $\mathcal{Y}_k^i(x)$ is chosen in such a way that $\mathcal{Y}_k^1(x_0)= \sqrt{\frac{d_k(n)}{|S^n|}}$ and $\mathcal{Y}_k^i(x_0)=0$ for all $i>0,$ then
$$ \sqrt{\frac{|S^n|}{d_k(n)}}\mathcal{Y}_k^i(x) = (\mathbb{P}_{SO(n)}\mathcal{T}_k^{i1})(x) = \int_{SO(n)} \mathcal{T}_k^{i1}(gh)\, dh = \mathcal{T}_k^{i1}(g),\quad x=gx_0.$$
\begin{theorem}
If $f$ belongs to $\mathcal{W}=\>span\>\{\mathcal{T}_{k}\}$, i.e. 
$
f(g) = \sum_{k=0}^{\infty}\sum_{i,j=1}^{d_k(n)} \hat{f}(k)_{ij} \mathcal{T}_k^{ij}
$
then
 $$
\mathcal{R} f(x,y)  =|S^n| \sum_{k=0}^{\infty} \sum_{i,j=1}^{d_k(n)} \hat{f}(k)_{ij} \mathcal{Y}_k^{i}(x)\overline{\mathcal{Y}_k^{j}(y)}.
$$
\end{theorem} 
\begin{proof}One has
\begin{equation}
\mathcal{R} f(x,y)= \sum_{k=0}^{\infty} d_k(n) {\rm trace\,}(\hat{f}(k)\mathcal{T}_k(x)\pi_{SO(n)}\mathcal{T}_k^*(y)) =
$$
$$
 \sum_{k=0}^{\infty} d_k(n) \sum_{i,j=1}^{d_k(n)} \hat{f}(k)_{ij} \mathcal{T}_k^{i1}(x)\overline{\mathcal{T}_k^{1j}(y)}
           = \sum_{k=0}^{\infty}\frac{|S^n|}{d_k(n)} d_k(n) \sum_{i,j=1}^{d_k(n)} \hat{f}(k)_{ij} \mathcal{Y}_k^{i}(x)\overline{\mathcal{Y}_k^{j}(y)}=
           $$
           $$
            |S^n| \sum_{k=0}^{\infty} \sum_{i,j=1}^{d_k(n)} \hat{f}(k)_{ij} \mathcal{Y}_k^{i}(x)\overline{\mathcal{Y}_k^{j}(y)}.
\end{equation}
\end{proof}

\section{Problem 2: Radon transform on $SO(3)$}\label{RTSO}

In this section we concentrate on the case $\mathcal{G} = SO(3),$ $\mathcal{H}= SO(2)$ and thus $\mathcal{G}/ \mathcal{H}= SO(3)/SO(2) = S^2.$ 
An orthonormal system in $L^2(S^2)$ is provided by the spherical harmonics $\{\mathcal{Y}_k^i,\,k\in \mathbb{N}_0,\ i=1,\ldots , 2k+1\}.$ The subspaces  $\mathcal{H}_k:={\rm span}\,\{\mathcal{Y}_k^i, i=1,\,\ldots ,\,2k+1\}$ spanned by the spherical harmonics of degree $k$ are the invariant subspaces of the quasi-regular representation 
$ T(g):\,f(x)\mapsto f(g^{-1}\cdot x), $
(where $\cdot $ denotes the canonical action of $SO(3)$ on $S^2$). Representation $T$ decomposes into $(2k+1)$-dimensional irreducible representation $\mathcal{T}_k$ in $\mathcal{H}_k.$
The corresponding matrix coefficients are the Wigner-polynomials
$$ \mathcal{T}_k^{ij}(g) = \langle \mathcal{T}_k(g)\mathcal{Y}_k^i, \mathcal{Y}_k^j \rangle. $$
If $\Delta_{SO(3)}$ and $\Delta_{S^{2}}$ are Laplace-Beltrami operators of invariant metrics on $SO(3)$ and $S^{2}$ respectively, then 
$$
 \Delta_{SO(3)}\mathcal{T}_k^{ij} = -k(k+1)\mathcal{T}_k^{ij}\quad \mbox{and}\quad \Delta_{S^2}\mathcal{Y}_k^i = -k(k+1)\mathcal{Y}_k^i. 
$$
 Using  the fact that $\Delta_{SO(3)}$ is equal to $-k(k+1)$ on the eigenspace $\mathcal{H}_k$ we obtain 
$$ ||f||^2_{L^2(SO(3))} = \sum_{k=1}^{\infty} (2k+1) ||\hat{f}(k)||^2_{HS} =$$
$$ \sum_{k=1}^{\infty} (2k+1) ||(4\pi)^{-1}\hat{f}(k)||^2_{L^2(S^2\times S^2)} = ||(4\pi)^{-1}(I-2\Delta_{S^{2}\times S^{2}})^{1/4}\mathcal{R} f||^2_{L^2(S^2\times S^2)} , $$
where $\Delta_{S^{2}\times S^{2}}=\Delta_1 + \Delta_2 $ is the Laplace-Beltrami  operator of the natural metric on $S^2\times S^2.$
We define the following norm on the space  $ C^{\infty}(S^2\times S^2)$
$$
 ||| u |||^2 = ((I-2\Delta_{S^2\times S^2} )^{1/2}u,\,u)_{L^2(S^2\times S^2)}.
 $$
Because $mathcal{R}$ is essentially an isometry between $L^{2}(SO(3))$ with the natural norm and $L^{2}(S^{2}\times S^{2})$ with the norm $|||\cdot|||$ the inverse of $\mathcal{R}$  is given by  its  adjoint operator. 
To calculate the adjoint operator we express the Radon transform $\mathcal{R}$ in another way. 
Going back to our problem in crystallography we first state that the great circle $C_{x,y}=\{g\in SO(3): g\cdot x = y\}$ in $SO(3)$ can also be described by the following formula
$$ C_{x,y}= x^{\prime}SO(2)(y^{\prime})^{-1}  := \{x^{\prime}h(y^{\prime})^{-1},\ h\in SO(2)\}, \quad x^{\prime},\,y^{\prime}\in SO(3), $$
where $x^{\prime}\cdot x_0 = x,\ y^{\prime}\cdot x_0 = y$ and $SO(2)$ is  the stabilizer of $x_0\in S^2.$ Hence,
\begin{eqnarray*}
\mathcal{R} f(x,y) = \int_{SO(2)} f(x^{\prime}h(y^{\prime})^{-1})\, dh = 4\pi \int_{C_{x,y}} f(g)\, dg  \\ =
 4\pi \int_{SO(3)} f(g)\delta_y(g^{-1}\cdot x)\, dg, \quad f\in L^2(SO(3)). 
 \end{eqnarray*}
To calculate the adjoint  operator we  use the last representation of $\mathcal{R}.$ We have
\begin{align*}
(\mathcal{R}^*u,\,f)_{L^2(SO(3))} & = ((I-2\Delta_{S^2\times S^2})^{1/2}u,\,\mathcal{R} f)_{L^2(S^2\times S^2)}= \\
                          &  (4\pi) \int_{S^2\times S^2} (I-2\Delta_{S^2\times S^2})u(x,y)\int_{SO(3)} f(g)\delta_y(g^{-1}\cdot x)\,dg\,dx\,dy= \\
                          &  (4\pi) \int_{SO(3)}\int_{S^2} (I-2\Delta_{S^2\times S^2})^{1/2}u(g\cdot y,\,y)\,dy\, f(g)\,dg,
\end{align*}
i.e. the $L^2$-adjoint operator is given by
\begin{eqnarray}\label{adjoint_op}
  \mathcal{R}^*u = (4\pi) \int_{S^2} (I-2\Delta_{S^2\times S^2})^{1/2}u(g\cdot
y,\,y)\,dy.
\end{eqnarray}

\begin{definition}[Sobolev spaces on $S^2\times S^2$] The Sobolev space 
  $H_{t}(S^2\times S^2),\,t\in \mathbb{R},$ is defined as the domain of the operator 
  $(I-2\Delta_{S^2\times S^2})^{\tfrac{t}{2}}$ with graph norm
		$$ ||f||_t = ||(I-2\Delta_{S^2\times S^2})^{\tfrac{t}{2}}f||_{L^2(S^2\times S^2)} ,
		$$
and 
the Sobolev space $H_t^{\Delta}(S^2 \times S^2),\,t\in \mathbb{R},$ is defined as the
 subspace of all functions $f\in H_t(S^2\times S^2)$ such $\Delta_1 f = \Delta_2 f.$ 
\end{definition}

\begin{definition}[Sobolev spaces on $SO(3)$] The Sobolev space $H_t(SO(3)),\,t \in \mathbb{R},$ is defined as the domain of the operator $(I-4\Delta_{SO(3)})^{\tfrac{t}{2}}$ with graph norm
			$$ |||f|||_t = ||(I-4\Delta_{SO(3)})^{\tfrac{t}{2}}f||_{L^2(SO(3))},\>\>f\in L^2(SO(3)). $$
\end{definition}	

\begin{theorem}
    For any $t\geq 0$ the Radon transform on $SO(3)$ is an invertible  mapping
        \begin{align}
           \mathcal{R} : H_{t}(SO(3))\to H_{t+\frac{1}{2}}^{\Delta}(S^2\times S^2).
        \end{align}
     and
 \begin{eqnarray}\label{inv_formel}
  f(g) = \int_{S^2} (I-2\Delta_{S^2\times S^2})^{\tfrac{1}{2}}(\mathcal{R} f)(gy, y) dy = \frac{1}{4\pi} (\mathcal{R} ^* \mathcal{R} f )(g).
 \end{eqnarray}  
\end{theorem}
\begin{proof}
For the mapping properties it is sufficient to consider case $t=0$. 
Because the Radon transform is an isometry up to the factor $4\pi$, we obtain
(\ref{inv_formel}).
\end{proof}

Since 
\begin{equation*}
\mathcal{R}(\mathcal{T}^k)(x,y)=\mathcal{T}^k(x)\pi_{SO(2)}\left(\mathcal{T}^k(y)\right)^*
\end{equation*}
we have
\begin{align*}
    \mathcal{R} \mathcal{T}^k_{ij}(x,y) &= \mathcal{T}_{i1}^k(x) \overline{\mathcal{T}^k_{j1}(y} )
    =\frac{4\pi}{2k+1} \mathcal{Y}_k^i(x)\overline{\mathcal{Y}_k^j(y)}.
   \end{align*}

\begin{theorem}[Reconstruction formula]\label{Recon}
Let
\begin{align*}
    G(x,y) = \mathcal{R} f(x,y) &= \sum_{k=0}^\infty \sum_{i,j=1}^{2k+1} \widehat G(k)_{ij} \mathcal{Y}_k^i(x) 
    \overline{\mathcal{Y}_k^j(y)} \in H_{\frac{1}{2}+t}^{\Delta}(S^2\times S^2),\
    t\geq 0,
\end{align*}
be a result of the Radon transform. Then the pre-image $f\in H_t(SO(3)),\ t\geq 0,$ is given by
\begin{align*}
    f &= \sum_{k=0}^\infty \sum_{i,j=1}^{2k+1} \frac{(2k+1)}{4\pi} \widehat G(k)_{ij} 
    \mathcal{T}_{ij}^k = \sum_{k=0}^\infty \sum_{i,j=1}^{2k+1} (2k+1)\widehat f(k)_{ij} 
    \mathcal{T}_{ij}^k  \\
    & = \sum_{k=0}^\infty (2k+1){\rm trace\,}(\widehat f(k) \mathcal{T}^k).
\end{align*}
\end{theorem}

\section{Radon transforms on the group $SO(3)$ and the sphere $S^3$}\label{relations1}

At the beginning of this section we show that  $S^3$ is a double cover of $SO(3)$. This fact allows us to identify  every  
function $f$ on $SO(3)$  with an even function on $S^3.$ After this  identification the  crystallographic Radon transform on $SO(3)$ becomes the geodesic Radon transform on $S^{3}$ in the sense of Helgason  \cite{H3}, \cite{SH1}, \cite{SH2}. In Theorem \ref{newinversion}  we show how Helgason's inversion formula for this transform can be  interpreted in crystallographic terms.

\subsection{Quaternions and rotations}
To understand the  crystallographic Radon transform one has to understand relations between $SO(3), \>S^{3}, \> S^{2}\times S^{2}$.  One of the ways to describe these relations is by using  the algebra of quaternions (see  \cite{Lounesto},   \cite{bernstein/schaeben}, \cite{P}). 

\begin{definition} Quaternions $\mathbb{H}$ are hypercomplex numbers of the form 
	$$ q = a_0+a_1i+a_2j+a_3k, $$
	where $a_0,\,a_1,\,a_2,\,a_3$ are real numbers and the generalized imaginary units $i,\,j,\,k$ satisfy the following multiplication rules:
\begin{align*}
	i^2 & = j^2 = k^2 = -1, \\
	ij & = k =-ji,\ jk = i = -kj,\ ki = j = -ik. 
\end{align*}
\end{definition}

\begin{definition} A quaternion $ q = a_0+a_1i+a_2j+a_3k = q_0 + \mathbf{q}$ is the sum of the real part $q_0 = a_0$ and the pure part $\mathbf{q}= a_1i+a_2j+a_3k .$ A quaternion $q$ is called pure if its real part vanishes. The conjugate $\bar{q}$ of a quaternion $q=a_0+\mathbf{q}$ is obtained by changing the sign of the pure part:
	$$ \bar{q} = a_0 - \mathbf{q}. $$
The norm $||q||$ of a quaternion $q$ is given by $||q||^2= q\bar{q}= a_0^2+a_1^2+a_2^2+a_3^2$ and coincises with the Euclidean norm of the associated element in $\mathbb{R}^4.$ 
\end{definition}
All non-zero quaternions are invertible with inverse $q^{-1}=\frac{\bar{q}}{||q||^2}.$
Next, we connect quaternions and rotations in $\mathbb{R}^3.$ Take a pure quaternion or a vector
$$ \mathbf{a}= a_1i+a_2j+a_3k \in \mathbb{R}^3  $$
with norm $||\mathbf{a}||= \sqrt{a_1^2+a_2^2+a_3^2}.$ For a non-zero quaternion $q\in\mathbb{H}$ the element $q\mathbf{a}q^{-1}$ is again a pure quaternion with same length, i.e. $||q\mathbf{a}q^{-1}||= ||\mathbf{a}||.$ That means that the mapping $\mathbb{R}^3\to \mathbb{R}^3$
$$ \mathbf{a} \mapsto q\mathbf{a}q^{-1} $$ is a rotation with the natural identification of $ \mathbb{R}^3$ with the set of pure quaternions. Each rotation in $SO(3)=\{U\in Mat(3,\mathbf{R}):\,U^TU = I,\ {\rm det\,}U = 1\}$ can be  represented in such form  and there are two unit quaterions $q$ and $-q$ representing the same rotation $q\mathbf{a}q^{-1}=(-q)\mathbf{a}(-q^{-1}).$ That means that
$$ S^3 = \{q\in \mathbb{H}: ||q||=1\} $$
is a two-fold covering group of $SO(3),$ i.e. $SO(3)\simeq S^3/\{\pm 1\}.$ 

\begin{definition}[\cite{bernstein/schaeben}] \label{definitionC}Let $q_1,\,q_2$ be two unit orthogonal quaternions, i.e. the scaler part of $q_1q_2$ which is equal to Euclidean scalar product of the vectors $q_1$ and $q_2$ is zero. The set of quaternions
$$ q(t) = q_1\cos t + q_2 \sin t,\quad t\in [0,\,2\pi) $$
is called a circle in the space of unit quaternions and denoted as $C_{q_1,q_2}.$
\end{definition}

Obviously, the circle $C_{q_1,q_2}$ is the intersection of the unit sphere $S^{3}$ with the plane $E(q_1,q_2)$ spanned by $q_1,\,q_2$ and passing though the origin $\mathcal{O}$.  
\begin{theorem} [\cite{bernstein/schaeben}] \label{rotations}Given a pair of unit vectors $(x,y)\in S^2\times S^2$ with $x\not=-y,$ the "great circle"  $C_{x,y} \in SO(3)$ of all rotations with $gy = x$ in $SO(3)$ may be represented as a great circle $C_{q_1,q_2}$ of unit quaternions such that
$$  C_{q_1,q_2}:= E(q_1,q_2)\cap S^3 $$
with 
\begin{equation}\label{relations} q_1:= \cos \frac{\eta}{2} + \frac{y\times x}{||y\times x||} \sin \frac{\eta}{2}, \quad q_2 :=  \frac{y+x}{||y+x||}, \end{equation}
where $\eta$ denotes the angle between $x$ and $y,$ i.e. $\cos \eta = y\cdot x.$ Rotation $gy = x$ in $SO(3)$ corresponds to rotation $x=qy\bar{q}$ in $ \mathbb{H}$.
\end{theorem}
For an arbitrary quaternion $q$  we define the linear map $\tau(q)$ of the algebra of quaternions $\mathbb{H}$ into itself which is given by the formula
\begin{equation}\label{cover}
\tau(q)h=qh\bar{q},\>h\in  \mathbb{H}.
\end{equation}
 One can check that if $q\in S^{3}$ then $\tau(q)\in SO(3)$.
 
Let  us summarize the following important facts (see \cite{Meister}, \cite{P} for more details). 
\begin{enumerate}

\item  The map $\tau: q\rightarrow \tau(q)$ has the  property $\tau(q)=\tau(-q)$ which shows that $\tau$ is a double cover of $S^{3}$ onto $SO(3)$.

\item $\tau$ maps 
\begin{equation}\label{mappingcircles}
\tau: C_{q_{1}, q_{2}}\rightarrow C_{x,y},
\end{equation}
where  $C_{q_{1}, q_{2}}= E(q_1,q_2)\cap S^3$ is a great circle in  $S^{3}$ and  $C_{x,y}$ is a great circle in $SO(3)$  of all rotations $g$  with $gy=x,\>(x,y)\in S^{2}\times S^{2}$ (relations between $(q_{1}, q_{2})$ and $(x,y)$ are given in (\ref{relations})). Conversely, pre-image of $C_{x,y}$ is $C_{q_{1},  q_{2}}$.

\item Great circles $C_{q_{1},q_{2}}$ are geodesics in $S^{3}$ in the natural metric.

\item The variety of all great circles $C_{x,y}\in SO(3),\>\>(x,y)\in S^{2}\times S^{2},$  (which are sets of all rotations $g$  with $gy=x$) can be identified with the product $S^{2}\times S^{2}$. For any $(x,y)\in S^{2}\times S^{2}$ the circles $C_{x,y}$ and $C_{-x,y},\>\>\>x\neq -y,$ are contained in orthogonal $2$-planes in $ \mathbb{H}$.
\end{enumerate}

\subsection{Radon transforms on $S^{3}$ and on $SO(3)$}

Let $\Xi$ denote the set of all $1$--dimensional geodesic submanifolds 
$\xi \subset S^3$.
According to the previous subsection each $\xi \in \Xi$ is a great circle of $S^{3}$, i.e. a circle with centre ${\mathcal O}$.  The manifold $\Xi$ can be identified with the manifold $S^{2}\times S^{2}$.

Following Helgason (see \cite{H3}, \cite{SH1}, \cite{SH2}), we introduce the next definition.

\begin{definition} 
For a continuous function $F$ defined on $S^{3}$ its $1$--dimensional
spherical (geodesic) Radon transform $\hat{ F}$ is a function, which is defined on any $1$-dimensional geodesic submanifold $\xi\subset S^{3}$ by the following formula 
\begin{equation}\label{Helg}
\hat{F}(\xi)=\frac{1}{2\pi} \int_{\xi} F(q) \, d\omega_1(q) = \int_{\xi} F(q) \, dm(q),
\end{equation}
with the normalized measure $m=\frac{1}{2\pi} \omega_1$ where $\omega_1$ denotes the usual one--dimensional
circular Riemannian measure.
\end{definition}
To invert transformation (\ref{Helg}) Helgason introduces dual transformation
\begin{equation}\label{dual}
\check{\phi}(q) = \int_{q\in \xi} \phi(\xi)\,d\mu(\xi), \>\>q\in S^{3},
\end{equation}
which represents  the average of a continuous function $\phi$  over all $\xi\in \Xi$ passing through $q\in S^{3}.$ Further,
$$ \check{\phi}_\rho(q) = \int_{\{d(q,\xi)= \rho\}} \phi(\xi) d\mu(\xi),\ \rho\geq 0, \>\>q\in S^{3},$$
where $d\mu$ is the average over the set of great circles $\xi$ at distance $\rho$ from $q.$ 
We use the inversion formula of  S. Helgason \cite{SH2}, which was obtained
for the general case two-point homogeneous spaces. For two dimensional
sphere the totally geodesic Radon transform is also known as the Funk transform. The inversion  formula  can be written as 
\begin{equation}\label{HelgInver}
 F(q) = \frac{1}{\pi} \left.\left[\frac{d}{du^2}\int_0^u (\hat{F})^{\check{}}_{cos^{-1}(v)}(q)v(u^2-v^2)^{-1/2}dv\right]\right|_{u=1},\>\>\>q\in S^{3}. 
 \end{equation}

Let us describe relations between geodesic Radon transform of functions defined on $S^{3}$ and the Radon transform $\mathcal{R}$ of functions defined on $SO(3)$. Given a function $f$ on $SO(3)$ one can  consider its Radon transform $\mathcal{R} f$ which is defined on the set of all great circles $C_{x,y}\subset SO(3)$.  On the other hand one can construct an even function $F$ on $S^{3}$ by using the formula 
\begin{equation}\label{pullback}
F(q)=f(\tau(q)),\>\>\>q\in S^{3},
\end{equation}
where the mapping $\tau: S^3 \to SO(3)$ was defined in (\ref{cover}).  For the function $F$ one can consider its geodesic Radon transform $\hat{F}$ which is defined on the set of all great circles $C_{q_{1}, q_{2}}\subset S^{3}$. One can check that the following formula holds
\begin{equation}\label{relation}
\mathcal{R} f(C_{x,y}) = \frac{1}{2\pi} \int_{C_{x,y}} f(g) d\omega(g) = \frac{1}{\pi} \int_{C_{q_1,q_2}} F(q)\,dq = 2\hat{F}(C_{q_1,q_2}), 
\end{equation}
where relations between circles $C_{x,y}$ and $C_{q_{1}, q_{2}}$ where described in Proposition \ref{rotations}. 
Since varieties of great circles on $S^{3}$ and on $SO(3)$ can be parametrized by points $(x,y)\in S^{2}\times S^{2}$ both transforms $\mathcal{R} f$ and $\hat{F}$ can be considered as functions on $S^{2}\times S^{2}$.

To describe  connection between different transforms and functions  it is useful to introduce  the
angle density function 
$$
({\mathcal A}F)({x}, {y}; \rho)=\frac{1}{2\pi}\int_{c(y; \rho)}\hat{F}(x,y^{\prime})d\omega_1(y^{\prime}),
$$
where $F$ is a function on $S^{3}$ and where $c(y; \rho)$ is a small circle of radius $\rho$ centered at  $y$. Note, that $({\mathcal A}F)({x}, {y}; \rho)$ was introduced in
\cite{HJB1} and  \cite{HJB2}.

The following properties hold
\begin{eqnarray}
({\mathcal A}F)({ x}, {y}; 0) & = & \hat{F}({ x}, {y}) ,\label{ar} \\
({\mathcal A}F)({ x}, { y}; \pi) & = & \hat{F} ({ x}, -{ y}) .\nonumber
\end{eqnarray}
According to its definition the quantity $({\mathcal A}F)({ x}, { y}; \pi) $ is the mean value of the spherical pole probability density function
over any small
circle centered at $y$. Thus, it is the probability density that the crystallographic direction 
${x}$
statistically encloses the angle $\rho, \>\>0 \leq \rho \leq \pi,$ with the specimen direction 
${ y}$ given the
orientation probability density function $F$. Its central role for the inverse
Radon transform was recognized in \cite{MEB} \cite{SM79}.

Our objective is to present two other inversion formulas. 

\begin{lemma}\label{formula}Let $F$ be an even continuous function on $S^3.$ Then the geodesic Radon transform $\hat{F}$ can be inverted by the following formula
\begin{equation}\label{secondinvert}
	F(q) = \frac{1}{2\pi} \left[(\hat{F})^{\check{}}_{\tfrac{\pi}{2}}(q) + 2 \int_0^{\pi}\left( \frac{d}{d\cos\theta}(\hat{F})^{\check{}}_{\tfrac{\theta}{2}}(q)\right)\cos\tfrac{\theta}{2}d\theta
	\right],\>\>q\in S^{3}. 	
\end{equation}
\end{lemma}
\begin{proof}
We start with $t=v^2$ to obtain
$$  F(q) = \frac{1}{2\pi} \left.\left[\frac{d}{du^2}\int_0^{u^2} (\hat{F})^{\check{}}_{cos^{-1}(\sqrt{t})}(q)\frac{1}{\sqrt{u^2-t}}dt\right]\right|_{u=1},  $$
and $s = u^2$
$$  F(q) = \frac{1}{2\pi} \left.\left[\frac{d}{ds}\int_0^{s} (\hat{F})^{\check{}}_{cos^{-1}(\sqrt{t})}(q)\frac{1}{\sqrt{s-t}}dt\right]\right|_{s=1},   $$
to shift the singularity inside the integral we set $\gamma=s-t$ which leads to
$$  F(q) = \frac{1}{2\pi} \left.\left[\frac{d}{ds}\int_0^{s} (\hat{F})^{\check{}}_{cos^{-1}(\sqrt{t})}(q)\frac{1}{\sqrt{\gamma}}d\gamma\right]\right|_{s=1},   $$
now we take the derivative
$$  F(q) = \frac{1}{2\pi} \left.\left[((\hat{F})^{\check{}}_{\cos^{-1}(0)}(q)\frac{1}{\sqrt{s}} + \int_0^{s} \frac{d}{ds}(\hat{F})^{\check{}}_{cos^{-1}(\sqrt{s-\gamma})}(q)\frac{1}{\sqrt{\gamma}}d\gamma\right]\right|_{s=1}. $$
Using 
$$\frac{d}{ds}(\hat{F})^{\check{}}_{cos^{-1}(\sqrt{s-\gamma})}(q) = -\frac{d}{d\gamma}(\hat{F})^{\check{}}_{cos^{-1}(\sqrt{s-\gamma})}(q) $$
and incorrporate $s=1$ we get
$$ F(q) = \frac{1}{2\pi} \left[(\hat{F})^{\check{}}_{\cos^{-1}(0)}(q) - \int_0^{1} \frac{d}{d\gamma}(\hat{F})^{\check{}}_{cos^{-1}(\sqrt{1-\gamma})}(q)\frac{1}{\sqrt{\gamma}}d\gamma\right]. $$
Substitution
$$ 2\gamma = 1-\cos\theta = 2\sin^2\tfrac{\theta}{2},\quad \sqrt{1-\gamma} = \cos\tfrac{\theta}{2}
 $$
 gives the formula (\ref{secondinvert}).
 Lemma is proved. 
 \end{proof}
 
 The formula in the next Theorem  coincides with an inversion formula which was reported  by S. Matthies in \cite{SM79} without any proof. 
The practical importance of this formula is that ${\mathcal A}F$ is easily experimentally accessible and might
yield an improved inversion algorithm.

\begin{theorem} \label{newinversion}
Suppose  that $f$ is a continuous function on $SO(3)$ and function $F$ on $S^{3}$ is defined according to (\ref{pullback}). Then the following reconstruction formula holds 
\begin{equation}\label{NewFormula}
f(g)=\frac{1}{4\pi}\int_{S^2} \hat{F}(x,-gx)d\omega_2(x) +
 $$
 $$
 \frac{1}{2\pi} 
 \int_0^{\pi} \int_{S^2} \frac{d}{d\cos\theta} (\mathcal{A}F)(x,gx;\theta)d\omega_2(x)\cos\tfrac{\theta}{2}d\theta , \>\>\>g\in SO(3),
\end{equation}
where  $\omega_2$ is the usual
two--dimensional spherical Riemann measure. 
\end{theorem}
\begin{proof}
According to Lemma \ref{formula}
we need to show that
\begin{align}
\int_{S^2} \hat{F}(x,-qx\bar{q})d\omega_2(x)  & = 2 (\hat{F})^{\check{}}_{\tfrac{\pi}{2}}(q) ,\label{R0}\\
\int_{S^2}  (\mathcal{A} F)(x, qx\bar{q};\theta)d\omega_2(x) & = 2 (\hat{F})^{\check{}}_{\tfrac{\theta}{2}}(q), \label{R1}
\end{align}
are fulfilled. Because (\ref{R0}) is a special case of (\ref{R1}) it is enough to verify the last equation. 
 For $g=\tau(q)$ we have
$$ \int_{S^2} (\mathcal{A} F)(x, qx\bar{q},\theta) d\omega_2(x) = 2\int_{\{d(g, \xi)=\tfrac{\theta}{2}\}} \hat{F}(\xi) d\mu(\xi) = 2 (\hat{F})^{\check{}}_{\tfrac{\theta}{2}}(q), $$
where $d\mu$ is the average over the set of $\xi$ at distance $\tfrac{\theta}{2}$  from $g=\tau(q).$ Since $\tau(q)x=qx\bar{q}=gq,\>\>g=\tau(q),$ we obtain the second formula.
Theorem is proved.
\end{proof}

\section{Problem 1: Inversion of crystallographic $X$-ray transform}\label{1}
Unfortunately, neither the Radon transform  $\mathcal{R} f$ over $SO(3)$ nor  the Radon
transform $\hat{f}$ over $S^3$ allows us to solve the crystallographic problem.
The point is that since
$$ \mathcal{Y}_k^i(-x) = (-1)^k\mathcal{Y}_k^i(x),
$$
one has for $\Phi(x,y)=\mathcal{R} f(x,y)$:
\begin{align*}
Pf(x,y)  & = \frac{1}{2}\left(\mathcal{R} f(x,y) + \mathcal{R} f(-x,y)\right)  = \frac{1}{2}(\Phi(x,y)  + \Phi(-x,y))\\ 
& = \frac{1}{2}
\left(\sum_{k=0}^\infty \sum_{i,j=1}^{2k+1} \widehat \Phi(k)_{ij} \mathcal{Y}_k^i(x) 
    \overline{\mathcal{Y}_k^j(y)} + \sum_{k=0}^\infty \sum_{i,j=1}^{2k+1} \widehat {\Phi}(k)_{ij} 
    \mathcal{Y}_k^i(-x) 
    \overline{\mathcal{Y}_k^j(y)} \right)\\  
   & = \frac{1}{2}\left(\sum_{k=0}^\infty \sum_{i,j=1}^{2k+1} \widehat \Phi(k)_{ij} \mathcal{Y}_k^i(x) 
    \overline{\mathcal{Y}_k^j(y)} + \sum_{k=0}^\infty (-1)^k\sum_{i,j=1}^{2k+1} \widehat \Phi(k)_{ij} 
    \mathcal{Y}_k^i(x) 
    \overline{\mathcal{Y}_k^j(y)} \right) \\ 
& = \sum_{l=0}^\infty \sum_{i,j=1}^{4l+1} \widehat \Phi(2l)_{ij} \mathcal{Y}_{2l}^i(x) 
    \overline{\mathcal{Y}_{2l}^j(y)} . 
  \end{align*}
In other words we loose half of the data needed for the reconstruction, because the experiment (which is measuring PDF $Pf$) 
only gives the even coefficients $\widehat {\Phi}(2l)_{ij}$.  

Since the odd Fourier coefficients
$\widehat {\Phi}(2l+1)_{ij} $  of the function $\Phi(x,y)=\mathcal{R} f(x,y)$ disappear  one cannot reconstruct the function $f(g), g\in SO(3), $ from
$Pf (x,y).$
Note that we have two additional conditions stemming from the fact that $f$ is a
probability distribution function: 
\begin{enumerate}
  \item  $f(g)\geq 0,$ 
  \item $\int_{SO(3)} f(g) dg = 1.$
\end{enumerate}

The second condition  is just a normalization, the first condition is less trivial.

We obviously can reconstruct the even part $f_e(g)$ from the even coefficients
$\widehat{\Phi}(2l)_{ij}$. In our future work we are planning to utilize properties (1) and (2) to obtain some information about the odd component of $f$.

\section{Problem 4: Exact reconstruction of  a bandlimited function $f$  on $SO(3)$ from  a finite number of samples of  $ \mathcal{R} f$}\label{Discrete}

It is clear that in practice one has to face situations described in the Problems 3 and 4. Concerning the Problem 3 we refer to \cite{CKT} where an approximate inverse was found using the language of Gabor frames.
A solution to the Problem 4 will be described in the present section.

Let $B((x,y),r)$ be a metric ball on $S^{2}\times S^{2}$ whose center is $(x,y)$ and
radius is $r$. As it is explained in Appendix there exists 
a natural number $N_{S^{2}\times S^{2}}$, such that  for any sufficiently small $\rho>0$
there exists a set of points $\{(x_{\nu},y_{\nu})\}\subset S^{2}\times S^{2}$ such that:
\begin{enumerate}
\item the balls $B((x_{\nu},y_{\nu}), \rho/4)$ are disjoint,

\item  the balls $B((x_{\nu},y_{\nu}), \rho/2)$ form a cover of $S^{2}\times S^{2}$,

\item  the multiplicity of the cover by balls $B((x_{\nu},y_{\nu}), \rho)$
is not greater than $N_{S^{2}\times S^{2}}.$
\end{enumerate}

Any set of points, which has properties (1)-(3) will be called a metric
$\rho$-lattice.

For an $\omega>0$ let us consider the space ${\bf E}_{\omega}(SO(3))$ of $\omega$-bandlimited functions on $SO(3)$ i.e. the span of all Wigner functions $\mathcal{T}_{ij}^{k}$ with $k(k+1)\leq \omega$.

In what follows $\mathcal{E}_{\omega}(S^{2}\times S^{2})$ will denote  the span in the space $L^{2}(S^{2}\times S^{2})$ of all $\mathcal{Y}_k^i(\xi)\overline{\mathcal{Y}_k^j(\eta)}$  with $k(k+1)\leq \omega$ . 

The goal of this section is to prove the following discrete  reconstruction formula (\ref{Exact}) for functions $f$ in ${\bf E}_{\omega}(SO(3))$, which uses only a \textit{finite} number of samples of $\mathcal{R} f$.

\begin{theorem}(Discrete Inversion Formula)\label{SamplingTh}

There exists a $C>0$ such that for any $\omega>0$, if 
$$
\rho=C(\omega+1)^{-1/2},
$$  
then
for any $\rho$-lattice $\{(x_{\nu}, y_{\nu})\}_{\nu=1}^{m_{\omega}}$ of $S^{2}\times S^{2}$ , there exist positive weights
$$
\mu_{\nu}\asymp \omega^{-2}, 
 $$
 such that for every  function $f$  in $ {\mathbf E}_{\omega}(SO(3))$ the  Fourier coefficients $c_{i,j}^{k}\left( \mathcal{R} f\right)$  of its  Radon transform, i.e.
 $$
  \mathcal{R} f(x,y)=\sum_{i,j,k} c_{i,j}^{k}\left( \mathcal{R} f\right)\mathcal{Y}^{i}_{k}(x)\overline{\mathcal{Y}^{j}_{k}}(y),\>\>\>\>\>\>\>\>\>k(k+1)\leq \omega,\>\>\>\>\>(x,y)\in S^{2}\times S^{2},
 $$
  are given by the formulas 
 \begin{equation}
c_{i,j}^{k}\left(\mathcal{R} f\right)=\sum_{\nu=1}^{m_{\omega}}\mu_{\nu}\left(\mathcal{R} f\right)(x_{\nu},y_{\nu})\mathcal{Y}^{i}_{k}(x_{\nu})\overline{\mathcal{Y}^{j}_{k}}(y_{\nu}),
\end{equation}
and there exists a constant $C$, which is independent on $\omega$ such that
$$
 dim\>\mathcal{E}_{24\omega}(S^{2}\times S^{2})\leq m_{\omega}\leq C dim\> \mathcal{E}_{24\omega}(S^{2}\times S^{2}).
$$
The function $f$ can be reconstructed by means of  the formula
\begin{equation}\label{Exact}
f(g)=\sum_{k}\sum_{i,j}^{2k+1}\frac{(2k+1)}{4\pi}c_{i,j}^{k}\left( \mathcal{R} f\right)\mathcal{T}_{k}^{i,j}(g),\>\>\>\>\>g\in SO(3),
\end{equation}
in which $k$ runs over all natural numbers  such that $k(k+1)\leq \omega$.
\end{theorem}

\begin{proof}

 As the formulas 

\begin{align}\label{eigenvalues}
    \Delta_{SO(3)} \mathcal{T}_k^{ij}=-k(k+1) \mathcal{T}_k^{ij} ,\>\>\> \Delta_{S^2}\mathcal{Y}_k^i=-k(k+1)\mathcal{Y}_k^i.
\end{align}
 and
 
\begin{align}\label{basis-action1}
   \mathcal{R} \mathcal{T}_k^{ij}(x,y)  =\frac{4\pi}{2k+1} \mathcal{Y}_k^i(x)\overline{\mathcal{Y}_k^j(y)}
\end{align} 
 show the Radon transform of a function $f\in  {\mathbf E}_{\omega}(SO(3))$ is $\omega$-bandlimited on $S^{2}\times S^{2}$ in the sense that its Fourier expansion involves only functions $\mathcal{Y}^{i}_{k}\overline{\mathcal{Y}^{j}_{k}}$ which are eigenfunctions of $\Delta_{S^{2}\times S^{2}}$ with eigenvalue $-k(k+1)$. Let $\mathcal{E}_{\omega}(S^{2}\times S^{2})$ be the span of $\mathcal{Y}_k^i(\xi)\overline{\mathcal{Y}_k^j(\eta)}$ with $k(k+1)\leq \omega$. Thus
 $$
  \mathcal{R}: {\bf E}_{\omega}(SO(3))\rightarrow \mathcal{E}_{\omega}(S^{2}\times S^{2}).
 $$

Let $\{(x_{1}, y_{1}),...,(x_{m}, y_{m})\}$ be a set  of pairs of points in $SO(3)$ and  $\mathcal{M}_\nu=x_\nu SO(2) y_\nu^{-1}$ are corresponding submanifolds of $ SO(3),\>\>\nu=1,...,m$.

For a function $f\in {\bf E}_{\omega}(SO(3))$ and a vector (of measurements) $v=\left(v_{\nu}\right)_{1}^{m}$ where
$$
v_{\nu}=\int_{\mathcal{M}_{\nu}}f,
$$
one has 
$$
  \mathcal{R} f(x_{\nu},y_{\nu})=v_{\nu}.
$$

We are going to find exact formulas for all Fourier coefficients of $ \mathcal{R} f\in \mathcal{E}_{\omega}(S^{2}\times S^{2})$ in terms of a finite set of measurement. 
Since $SO(3)$ has dimension three then according to Theorem \ref{prodthm} (see Appendix) every product 
$
 \left( \mathcal{R} f\right)\mathcal{Y}^{i}_{k}\overline{\mathcal{Y}^{j}_{k}}$, where $k(k+1)\leq \omega$ belongs to   $\mathcal{E}_{\Omega}(S^{2}\times S^{2})$, where $\Omega=4\times 6\omega=24\omega$.

By the Theorem \ref{cubformula} (see Appendix) there exists  a  positive constant $C$,    such  that if  $\rho=C(\omega+1)^{-1/2}$, then
for any $\rho$-lattice $\{(x_{1}, y_{1}),...,(x_{m_{\omega}}, y_{m_{\omega}},)\}$ in $S^{2}\times S^{2}$ 
there exist  a set of positive weights 
$
\mu_{\nu}\asymp \Omega^{-2}
$
such that 
\begin{equation}
c_{i,j}^{k}\left( \mathcal{R} f\right)=\int_{S^{2}\times S^{2}} \left( \mathcal{R} f\right)(x,y)\mathcal{Y}^{i}_{k}(x)\overline{\mathcal{Y}^{j}_{k}}(y) dxdy=
$$
$$
\sum_{\nu=1}^{N}\mu_{\nu}\left( \mathcal{R} f\right)(x_{\nu},y_{\nu})\mathcal{Y}^{i}_{k}(x_{\nu})\overline{\mathcal{Y}^{j}_{k}}(y_{\nu}).
\end{equation}
Thus,
$$
\left( \mathcal{R} f\right)(x,y)=\sum_{\nu}c_{i,j}^{k}\left( \mathcal{R} f\right)\mathcal{Y}^{i}_{k}(x)\overline{\mathcal{Y}^{j}_{k}}(y),
$$
where
\begin{equation}\label{coef}
c_{i,j}^{k}\left( \mathcal{R} f\right)=\sum_{\nu=1}^{N}\mu_{\nu}\left( \mathcal{R} f\right)(x_{\nu})\mathcal{Y}^{i}_{k}(x_{\nu})\overline{\mathcal{Y}^{j}_{k}}(x_{\nu}).
\end{equation}

Now the reconstruction formula of Theorem \ref{Recon} gives our result (\ref{Exact}).

\end{proof}

\section{Appendix}\label{A}

We explain Theorems   \ref{cubformula}  and \ref{prodthm}, which played the key role in the proof of Theorem \ref{SamplingTh}.

\subsection{Positive cubature formulas on compact manifolds}

 We consider a compact connected Riemannian manifold ${\bf M}$. 
Let $B(\xi,r)$ be a metric ball on  ${\bf M}$ whose center is $\xi$ and
radius is $r$. 

It was shown in \cite{Pes00}, \cite{Pes04a}, that if ${\bf M}$ is compact then  there exists
a natural number $N_{{\bf M}}$, such that  for any sufficiently small $\rho>0$
there exists a set of points $\{\xi_{k}\}$ such that:
(1)  the balls $B(\xi_{k}, \rho/4)$ are disjoint;
(2) the balls $B(\xi_{k}, \rho/2)$ form a cover of ${\bf M}$; (3) the multiplicity of the cover by balls $B(\xi_{k}, \rho)$
is not greater than $N_{{\bf M}}.$

Any set of points $M_{\rho}=\{\xi_{k}\}$ which has properties (1)-(3) will be called a metric
$\rho$-lattice.

Let $L$ be an elliptic  second order differential operator on ${\bf M}$, which is self-adjoint and positive semi-definite in the space $L^{2}({\bf M})$ constructed with respect to Riemannian measure. Such operator has a discrete spectrum $0<\lambda_{1}\leq \lambda_{2}\leq .... $ which goes to infinity and does not have accumulation points. Let $\{u_{j}\}$ be an orthonormal system of eigenvectors of $L$, which is complete in $L^{2}({\bf M})$. 

For a given $\omega>0$ the notation $ {\mathbf E}_{\omega}({L})$ will be used for the span of all eigenvectors $u_{j}$ that correspond to eigenvalues not greater than $\omega$.

Now we are going to prove existence of cubature formulas which are exact on $ {\mathbf E}_{\omega}(L)$,
and have positive coefficients of the "right" size.

The following exact cubature formula was established  in \cite{gp}, \cite{pg}.
\begin{theorem} 
\label{cubformula}
There exists  a  positive constant $C$,    such  that if  
\begin{equation}\label{density}
\rho=C(\omega+1)^{-1/2},
\end{equation}
 then
for any $\rho$-lattice $M_{\rho}=\{\xi_{k}\}$, there exist strictly positive coefficients $\mu_{\xi_{k}}>0, 
 \  \xi_{k}\in M_{\rho}$, \  for which the following equality holds for all functions in $ {\mathbf E}_{\omega}({L})$:
\begin{equation}
\label{cubway}
\int_{{\bf M}}fdx=\sum_{\xi_{k}\in M_{\rho}}\mu_{\xi_{k}}f(\xi_{k}).
\end{equation}
Moreover, there exists constants  $\  c_{1}, \  c_{2}, $  such that  the following inequalities hold:
$$
c_{1}\rho^{n}\leq \mu_{\xi_{k}}\leq c_{2}\rho^{n}, \ n=dim\   {\bf M}.
$$
\end{theorem}
It is worth to noting that this result is essentially optimal in the sense that (\ref{density}) and Weyl's asymptotic formula $\mathcal{N}_{\omega}(L)\asymp C_{{\bf M}}\omega^{n/2},$ for the number of eigenvalues of $L$ imply  that cardinality of $M_{\rho}$  has the same order as dimension of the space $ {\mathbf E}_{\omega}({L})$.

\subsection{On the product of eigenfunctions of the Casimir operator $\mathcal{L}$ on compact
 homogeneous manifolds}

A homogeneous compact manifold ${\bf M}$ is a
$C^{\infty}$-compact manifold on which a compact
Lie group $\mathcal{G}$ acts transitively. In this case ${\bf M}$ is necessarily of the form $\mathcal{G}/\mathcal{H}$,
where $\mathcal{H}$ is a closed subgroup of $\mathcal{G}$. The notation $L^{2}({\bf M}),
$ is used for the usual Hilbert  spaces
$L^{2}({\bf M})=L^{2}({\bf M},d\xi)$, where $d\xi$ is an invariant
measure.

If $\textbf{g}$ is the Lie algebra of a compact Lie group $\mathcal{G}$ then
(\cite{H3}, Ch.\ II,) it is a direct sum
$\textbf{g}=\textbf{a}+[\textbf{g},\textbf{g}]$, where
$\textbf{a}$ is the center of $\textbf{g}$, and
$[\textbf{g},\textbf{g}]$ is a semi-simple algebra. Let $Q$ be a
positive-definite quadratic form on $\textbf{g}$ which, on
$[\textbf{g},\textbf{g}]$, is opposite to the Killing form. 
Let
$X_{1},...,X_{d}$ be a basis of
$\textbf{g}$, which is orthonormal with respect to $Q$. By using differential of the quasi-regular representation of $\mathcal{G}$ in the space $L^{2}({\bf M})$ one can identify every  $X_{j},\>j=1,...,d,$ with a first-order differential operator $D_{j},\>j=1,...,d,$ in the space $L^{2}({\bf M})$. 
 Since the form $Q$ is $Ad(\mathcal{G})$-invariant, the operator
$
\mathcal{L}=-D_{1}^{2}- D_{2}^{2}- ...- D_{d}^{2}, \>\>\>
            d=dim \ \mathcal{G},\label{Laplacian}
$
commutes with all operators $D_{j},\>j=1,...,d$. 

This elliptic second order differential operator
$\mathcal{L}$ is usually called the Laplace operator. In the case of a compact semi-simple
Lie group, or a compact symmetric space of rank one,  the operator $\mathcal{L}$ is proportional to the
Laplace-Beltrami operator  of an invariant metric on ${\bf M}$.

The following theorem was proved in \cite{gp}, \cite{pg}.

\begin{theorem}
\label{prodthm}
If ${\bf M}=\mathcal{G}/\mathcal{H}$ is a compact homogeneous manifold and $\mathcal{L}$
is defined as in (\ref{Laplacian}), then for any $f$ and $g$ belonging
to ${\mathbf E}_{\omega}(\mathcal{L})$,  their product $fg$ belongs to
${\mathbf E}_{4d\omega}(\mathcal{L})$, where $d$ is the dimension of the
group $\mathcal{G}$.

\end{theorem}

{\bf Acknowledgment.} We thank  the anonymous referee for encouraging us to improve the original manuscript,  and  Meyer Pesenson for helping us to address some of the referee's concerns.


\bibliographystyle{amsplain}

\begin{thebibliography}{99}
	

\bibitem{Asgeirsson}  Asgeirsson, L., {\em Uber eine Mittelwerteigenschaft von Losungen homogener linearer partieller Differentialgleichungen zweiter Ordnung mit konstanten Koeffizienten,} Annals of Mathematics 1937; 113, 312-346.

\bibitem{BBP}
Berens, H., Butzer, P., Pawelke, S.,  {\em Limitierungsverfahren mehrdimensionaler
Kugelfunktionen und deren Saturationsvehalten},  Publ. Res. Inst. Math. Sci., Kyoto Univ., Ser.
A, 4, 201--268 (1968),

\bibitem{BEP}
Bernstein, S., Ebert, S., Pesenson, I., {\em Generalized Splines for Radon Transform
on Compact Lie Groups with Applications to Crystallography,} J. Fourier Anal.
Appl., doi 10.1007/s00041-012-9241-6,

\bibitem{bernstein/schaeben}
Bernstein, S., Schaeben, H., {\em A one-dimensional Radon transform on $SO(3)$ and
its application to texture goniometry}, Math. Meth. Appl. Sci.,
28:1269--1289 (2005),

\bibitem{BHS}
Bernstein, S., Hielscher, R., Schaeben, H., {\em The generalized totally geodesic
Radon transform and its application to texture analysis}, Math. Meth. Appl.
Sci., 32:379--394 (2009),


\bibitem{BPS} K.G.van den Boogaart, R. Hielscher, J. Prestin and H. Schaeben, {\em  Kernel-based methods for inversion of the Radon
transform on SO(3) and their applications to texture analysis},  J. Comput. Appl. Math. 199 (2007),122-40,


\bibitem{HJB1}
Bunge, H.-J., {\em Mathematische Methoden der Texturanalyse}: Akademie Verlag, Berlin
(1969),

\bibitem{HJB2}
Bunge, H.-J., Morris, P.R., {\em Texture Analysis in Materials Science -- Mathematical
Methods}: Butterworths (1982),


\bibitem{CKT}
Cerejeiras, P., Ferreira, M., K\"ahler, U., Teschke, G., {\em Inversion of the noisy
Radon transform on $SO(3)$ by Gabor frames and sparse recovery principles}, Appl. Comput.
Harmon. Anal., 31, 325--396 (2011),

\bibitem{WF}
Freeden, W., Gervens, T., Schreiner, M., {\em Constructive Approximation on the Sphere
with Applications to Geomathematics}: Oxford Science Publication (1998)

\bibitem{F}
Friedel G.,
 {\em Sur les sym\'etries cristallines que peut r\'ev\'eler la diffraction des rayons X}, C.R. Acad. Sci. Paris, 157, 1533-1536 (1913),

\bibitem{gp}
Geller D.,  Pesenson I, {\em Band-limited localized Parseval frames and Besov spaces on compact homogeneous manifolds}, J. of Geometric Analysis, 21 (2011), no. 2, 334-371,


\bibitem {H3}
Helgason, S., {\em Differential Geometry and Symmetric Spaces},
Academic, N.Y., 1962,

\bibitem{SH1}
Helgason, S., 1994, {\em Geometric Analysis on Symmetric Spaces}, Mathematical Surveys and
Monographs 39, American Mathematical Society (1994),

\bibitem{SH2}
Helgason, S., {\em The Radon Transform, 2nd ed.}, Birkh\"auser (1999),


	\bibitem{hielscher}  Hielscher, R., Die Radontransformation auf der Drehgruppe -- Inversion und Anwendung in der Texturanalyse. PhD thesis, University of Mining and Technology Freiberg, 2007, 
	
	\bibitem{HPPSS}
	Hielscher, R., Potts, D., Prestin, J., Schaeben, H., Schmalz, M., {\em The Radon transform on SO(3): a Fourier slice theorem and numerical inversion}, Inverse Problems 24 (2008), no. 2, 025011, 21 pp.,


\bibitem{K}
Kakehi, T., Tsukamoto, C.,{\em Characterization of images of Radon transform},  Adv. Stud. Pure Math. (1993), 22, 101-16,	

\bibitem{Lounesto} Lounesto, P., {\em Clifford Algebras and Spinors,} London Mathematical Society Lecture Notes Series 286, 2nd edition, Cambridge University Press, 2001,

\bibitem{SM79}
Matthies, S., {\em On the reproducibility of the orientation distribution function of texture
samples from pole figures (ghost phenomena)}, Phys. Stat. Sol. (b), 92, K135--K138
(1979),

\bibitem{Matthies}
Matthies, S., {\em Aktuelle Probleme der quantitativen Texturanalyse,} ZfK-480. Zentralinstitut f\"ur Kernforschung Rossendorf bei Dresden, ISSN 0138-2950, August 1982,

\bibitem{Meister} Meister, L., Schaeben, H., {\em A coincise quaternionic geometry of rotations.} Mathematical Methods in the Applied Sciences, \textbf{28}: 101-126, 2004,


\bibitem{MEB}
Muller, J., Esling, C., Bunge, H.-J., {\em  An inversion formula expressing the texture
function in terms of angular distribution function}, J. Phys. 42, 161--165
(1981),

\bibitem{Naimark} Naimark, M. A., {\em Linear Representations of the Lorentz Group,} Pergamon Press, 1964,

\bibitem{P}
Palamodov, V.P., {\em Reconstruction from a sampling of circle integrals in SO(3)},  Inverse Problems 26 (2010), no. 9, 095-008, 10 pp.,

\bibitem{Pes00}
Pesenson, I., {\em A sampling theorem on homogeneous manifolds},
Trans. Amer. Math. Soc. {\bf 352} (2000), no. 9, 4257--4269,

\bibitem{Pes04a}
Pesenson,I., {\em An approach to spectral problems on Riemannian
manifolds,}  Pacific J. of Math. Vol. 215(1), (2004), 183-199,

\bibitem{pg}
Pesenson I., Geller D.,  {\em Cubature formulas and discrete Fourier transform on compact manifolds},  in  "From Fourier Analysis and Number Theory to Radon Transforms and Geometry: In Memory of Leon Ehrenpreis" (Developments in Mathematics) by Hershel M. Farkas, Robert C. Gunning, Marvin I. Knopp and B. A. Taylor (2012),


\bibitem{Taylor}
Taylor, M. E., {\em Noncommutative Harmonic Analysis}, vol. 22, Math. Surveys
and Monographs, AMS (1986),


\bibitem{Vilenkin} Vilenkin, N.J., Klimyk, A.U., {\em Representations of Lie Groups and special 
	functions}, volume~2, Kluwer Academic Publishers (1993),
	
\bibitem{vilenkin1} Vilenkin, N.J., {\em Special Functions and the Theory of Group Representations}, 
Translations of Mathematical Monographs Vol. 22, American Mathematical Society (1978).


\end{thebibliography}

\end{document}